\documentclass[12pt]{article}
\usepackage{amsmath}
\usepackage{float}
\usepackage{amsthm}
\usepackage{amssymb}
\usepackage[utf8x]{inputenc}
\usepackage{caption}
\usepackage{graphics}
\usepackage{enumerate} 
\usepackage{algorithmic}
\usepackage[Algorithm,ruled]{algorithm}
\usepackage{titlesec}
\usepackage{sectsty}
\titleformat*{\section}{\LARGE\bfseries}
\titleformat*{\subsection}{\Large\bfseries}
\titleformat*{\subsubsection}{\large\bfseries}
\sectionfont{\Huge}
\subsectionfont{\large}
\subsubsectionfont{\normalsize}
\usepackage{multicol}
\usepackage{lmodern}
\usepackage{marvosym} 
\usepackage{lipsum}
\usepackage{mwe}
\usepackage{caption}
\usepackage{subcaption} 
\newtheoremstyle{case}{}{}{}{}{}{:}{ }{}
\theoremstyle{case}

\newcommand{\be}{\begin{equation}}
\newcommand{\ee}{\end{equation}}
\newcommand{\ben}{\begin{eqnarray*}}
\newcommand{\een}{\end{eqnarray*}}
\newtheorem{examp}{\sc example}
\newtheorem{remk}{\sc remark}
\newtheorem{corol}{\sc corollary}
\newtheorem{lemma}{\sc lemma}
\newtheorem{theorem}{\sc theorem}
\newtheorem{defn}{\sc definition}
\newcommand{\bt}{\begin{theorem}}
\newcommand{\et}{\end{theorem}}
\newcommand{\bl}{\begin{lemma}}
\newcommand{\el}{\end{lemma}}
\newcommand{\bed}{\begin{defn}}
\newcommand{\eed}{\end{defn}}
\newcommand{\brem}{\begin{remk}}
\newcommand{\erem}{\end{remk}}
\newcommand{\bex}{\begin{examp}}
\newcommand{\eex}{\end{examp}}
\newcommand{\bcl}{\begin{corol}}
\newcommand{\ecl}{\end{corol}}

\newcommand{\NI}{\noindent}

\theoremstyle{definition}
\theoremstyle{remark}

\numberwithin{equation}{section}
\numberwithin{theorem}{section}
\numberwithin{lemma}{section}

\begin{document}

\title{\large\bf\sc On sparse solution of tensor complementarity problem}

\author{R. Deb$^{a,1}$ and A. K. Das$^{b,2}$\\
\emph{\small $^{a}$Jadavpur University, Kolkata , 700 032, India.}\\	
\emph{\small $^{b}$Indian Statistical Institute, 203 B. T.
	Road, Kolkata, 700 108, India.}\\
\emph{\small $^{1}$Email: rony.knc.ju@gmail.com}\\
%\emph{\small $^{2}$Email: aritradutta001@gmail.com}\\
\emph{\small $^{2}$Email: akdas@isical.ac.in}\\
}

\date{}

\maketitle

\begin{abstract}
\NI In this article we consider the sparse solutions of the tensor complementarity problem (TCP) which are the solutions of the smallest cardinality. We establish a connection between the least element of the feasible solution set of TCP and sparse solution for $Z$-tensor. We propose a $p$ norm regularized minimization model when $p\in (0,1)$ and show that it can approximate sparse solution applying the regularization of parameter.\\

\noindent{\bf Keywords:} Tensor complementarity problem, sparse solution, $l_p$ regularized minimization, $Z$-tensor.

\noindent{\bf AMS subject classifications:} 15A69, 90C33. 
\end{abstract}
\footnotetext[1]{Corresponding author}

\section{Introduction}
Huang and Qi \cite{huang2017formulating} reformulated the multilinear game as a tensor complementarity problem and showed that finding a Nash equilibrium point of the multilinear game is equivalent to finding a solution of the resulted tensor complementarity problem which built a bridge between these two classes of problems. In order to reduce the computational time for large-scale noncooperative multilinear games we get the motivation to study the srarse solution of tensor complementarity problem. The tensor complementarity problem denoted by TCP$(q,\mathcal{A})$ was introduced by Song and Qi \cite{song2017properties} which is a class of nonlinear complementarity problem with the function in terms of a tensor. Recently, a few special structured tensors have been investigated in the literature. For details see \cite{song2017properties}, \cite{song2015properties}.

\NI Let $\mathcal{A}$ be a tensor of order $m$ and dimension $n.$ i.e., $\mathcal{A}\in T_{m,n}.$ Given a vector $p\in \mathbb{R}^n,$ the tensor complementarity problem, denoted by TCP$(q,\mathcal{A})$ is to find $u \in \mathbb{R}^n$ such that, \begin{equation}\label{ Tensor Complementarity problem}
		u\geq 0, \;\;\;\; \mathcal{A}u^{m-1} + q\geq 0, \;\;\;\; \mbox{and}\;\;\;\; u^{T}(\mathcal{A}u^{m-1}+q)=0.
\end{equation}
When the order of the tensor $m=2$ then the problem reduces to a linear complementarity problem. Let $A$ be an $n\times n$ real matrix. Given a vector $q\in \mathbb{R}^n,$ the linear complementarity problem, denoted by LCP$(q,A)$ is finding $u \in \mathbb{R}^n$ such that,
\begin{equation}\label{linear comp problem}
		u\geq 0, \;\;\;\; q + A u \geq 0, \;\;\;\; \mbox{and}\;\;\;\; u^{T}(q+Au)=0.
	\end{equation}

\NI The idea of complementarity generates a large number of optimization problems. The problems which can be constituted as linear complementarity problem includes linear programming, linear fractional programming, convex quadratic programming and the bimatrix game problem. It is well considered in the literature on mathematical programming and occurs in a number of applications in operations research, control theory, mathematical economics, geometry and engineering. For recent works on this problem and applications see \cite{mohan2001more}, \cite{mohan2001classes}, \cite{neogy2006some}, \cite{neogy2005almost}, \cite{mohan2004note}, \cite{das2018invex}, \cite{dutta2023some}, \cite{jana2021iterative}, \cite{jana2018semimonotone} and references therein.
The algorithm developed by Lemke and Howson to work out an equilibrium pair of strategies to a bimatrix game, later modified by Lemke to find an LCP(q, M) provided remarkably to the development of the linear complementarity theory and established further research on the LCP. However, this algorithm does not solve every instance of the linear complementarity problem. In some instances, the problem may terminate inconclusively without either computing a solution to it or showing that no solution exists. For details see \cite{neogy2008mixture} and \cite{neogy2016optimization}.

\NI The concept of PPT appeared in the literature for more than four decades under different names and it is originally motivated by the well-known linear complementarity problem, and applied in many other settings. PPT is basically a transformation of the matrix of a linear system for exchanging unknowns with the corresponding entries of the right hand side of the system. For details see \cite{das2017finiteness}, \cite{mondal2016discounted}, \cite{neogy2012generalized}, \cite{das2016properties} and \cite{neogy2005principal}.

\NI Several matrix classes and their subclasses have been studied extensively because of their predominance in scientific computing, complexity theory, and the theoretical foundations of linear complementarity problems. For recent work on this problem and applications see \cite{jana2019hidden}, \cite{das2016generalized}, \cite{dutta2022on} and \cite{neogy2011singular}. 

\NI A large subclass of GPSBD matrices is identified as row sufficient matrices. This has practical relevance to the study of quadratic programming. In addition, the applicability of Lemke’s algorithm extends the class of LCP problems solvable by Lemke’s algorithm. For details see \cite{neogy2009modeling}, \cite{das2018some} and \cite{jana2018processability}. 
The class of N and N0-matrices arises in the theory of global univalence of functions, multivariate analysis and in complementarity theory. See \cite{mondal2016discounted}, \cite{jana2021more} and \cite{jana2018processability}.

\NI Another major areas of research in this field is to identify those classes of zero-sum stochastic games for which there is a possibility of obtaining a finite step algorithm to compute a solution. The class of structured stochastic game contains single control game, switching control game, games with state independent transitions and separable rewards and the games with additive reward and transitions. The problem of computing the value vector and optimal stationary strategies for structured stochastic games is formulated as a linear complementary problem for discounted and undiscounded zero-sum games. This provides an alternative proof of the orderfield property for these two classes of games. For details, see \cite{neogy2013weak}, \cite{neogy2008mathematical} and \cite{neogy2005linear} and references cited therein. 

\NI Sparse solution in LCP was studied by Chen and Xiang \cite{chen2016sparse} and Shang et al \cite{shang2014minimal}. The structure of tensors and polynomial properties defined by the tensors involved in the corresponding problems play crucial roles in the TCP. For the theories related to the existence of solutions and the compactness of the solution set, see \cite{che2016positive}, \cite{song2016properties}, \cite{ding2018p}, \cite{wang2016exceptionally}, \cite{gowda2015z} and \cite{wang2018solution}.
The class of strong $P$ tensor and strong strictly semipositive tensor are important due to global uniqueness of TCP. For the global uniqueness of the solution of the tensor complementarity problem see \cite{liu2018tensor} and \cite{bai2016global}. Finding the sparsest solutions to a tensor complementarity problem was considerd by Luo et al \cite{luo2017sparsest} for the class of $Z$-tensors.

The paper is organised as follows. Section 2 presents some basic notations and results. In section 3, We establish a connection between the least element of FEA$(q,\mathcal{A})$ and sparse solution of TCP$(q,\mathcal{A})$ for $\mathcal{A}$ being a $Z$ tensor. We introduce an unconstrained $p$ norm regularized minimization problem. We show that this minimization problem approximates the sparse solutions of the TCP. We give a lower bound on nonzero-entry of local minimizers and investigate the regularisation of parameter selection in our proposed regularised model.

\section{Preliminaries}

We introduce some basic notations used in this paper. We consider tensor, matrices and vectors with real entries. For any positive integer $n,$  $[n]$ denotes set $\{ 1, 2,...,n \}$. Let $\mathbb{R}^n$ denote the $n$-dimensional Euclidean space and $\mathbb{R}^n_+ =\{ u\in \mathbb{R}^n : u\geq 0 \}.$ Any vector $u\in \mathbb{R}^n$ is a column vector unless specified otherwise. The norm of a vector $u$ is defined as $\|u\| = \sqrt{|u_1^2| + \cdots |u_n^2|}.$ A $m$th order $n$ dimensional real tensor $\mathcal{A}= (a_{i_1 i_2 ... i_m}) $ is a multidimensional array of entries $a_{i_1 i_2 ... i_m} \in \mathbb{R}$ where $i_j \in [n]$ with $j\in [m]$. The set of all $m$th order $n$ dimensional real tensors are denoted by $T_{m,n}.$ For a tensor $\mathcal{A}\in T_{m,n}$ and $u\in \mathbb{R}^n,\; \mathcal{A}u^{m-1}\in \mathbb{R}^n $ is a vector defined by
	\[ (\mathcal{A} u^{m-1})_i = \sum_{i_2, ...,i_m =1}^{n} a_{i i_2 ...i_m} u_{i_2} \cdots u_{i_m} , \;\forall \; i \in [n], \]
	and $\mathcal{A}u^m\in \mathbb{R} $ is a scalar defined by
 \begin{equation*}
     u^T \mathcal{A}u^{m-1} = \mathcal{A}u^m = \sum_{i_1,...,i_m =1}^{n} a_{i_1  ...i_m} u_{i_1}  \cdots u_{i_m}.
 \end{equation*}
	
\NI Shao \cite{shao2013general} introduced product of tensors. Let $\mathcal{A}$ with order $q \geq 2$ and $\mathcal{B}$ with order $k \geq 1$ be two $n$-dimensional tensors. The product of $\mathcal{A}$ and $\mathcal{B}$ is a tensor $\mathcal{C}$ of order $(q-1)(k-1) + 1$ and dimension $n$ with entries $c_{i \alpha_1 \cdots \alpha_{m-1} } =\sum_{i_2, \cdots ,i_m \in [n]} a_{i i_2 \cdots i_m} b_{i_2 \alpha_1} \cdots b_{i_m \alpha_{m-1}},$ where $i \in [n] $, $\alpha_1, \cdots, \alpha_{m-1} \in [n]^{k-1}.$ %This product was proved to be associative by Shao (\cite{shao2013general}).

%\begin{defn}\cite{song2015properties}
%	Let $\mathcal{A}= (a_{i_1 i_2 i_3 ...i_m}) \in T_{m,n}$ and $I\subseteq [n],$ where $\lvert I \rvert=r,$ and $r\in [n].$ A principal subtensor of $\mathcal{A}$ is denoted by $\mathcal{A}^I_r$ and is defined as $ \mathcal{A}^I_r = ( a_{i_1 i_2 ...i_m} ), \; \forall \; i_1, i_2,...i_m \in I .$
%\end{defn}

\noindent Given a tensor $\mathcal{A}= (a_{i_1 ... i_m}) \in T_{m,n}$ and a vector $q \in \mathbb{R}^n$, we define the set of feasible solution of TCP$(q,\mathcal{A})$ as FEA$(q,\mathcal{A})= \{u\in \mathbb{R}^n : u\geq 0,\; \mathcal{A}u^{m-1} +p \geq 0\}$ and the solution set of TCP$(q,\mathcal{A})$ as SOL$(q,\mathcal{A})= \{u\in \mathbb{R}^n : u\geq 0,\; \mathcal{A}u^{m-1} +p \geq 0 \mbox{ and } u^{T}(\mathcal{A}u^{m-1}+p)= 0\}.$

\begin{defn}
	\cite{song2015properties} A tensor $\mathcal{A}= (a_{i_1 i_2 ... i_m}) \in T_{m,n} $ is said to be a $P(P_0)$-tensor, if for each $u\in \mathbb{R}^n \backslash \{0\}$, $\exists \; i\in [n]$ such that $u_i \neq 0$ and $u_i (\mathcal{A}u^{m-1})_i > (\geq 0)$.
\end{defn}

\NI For any real $x,$
$sign(u)=\left\{\begin{array}{cc}
    1 & \mbox{if } x>0; \\
    0 & \mbox{if } x=0; \\
    -1 & \mbox{if } x<0.
\end{array}\right.$
For any $p>0$ and $u\in \mathbb{R}^n,$ let
\begin{equation*}
    \|u\|_p = \left(\sum_{j=1}^n |u|_j^p \right)^{\frac{1}{p}} \mbox{ and } \|u\|_0= \lim_{p\to 0^+} \|u\|_p = \sum_{j=1}^n sign(|u_j|)
\end{equation*}
which is equal to the cardinality of $u.$ If $p\geq 1,$ the $p$-norm is denoted as $\|u\|_p$ for $u\in \mathbb{R}^n.$ Though for $0\leq p<1$, $\|u\|_p$ does not satisfy all the properties of norm.

\NI Given a  tensor $\mathcal{A} \in T_{m,n},$ the least $p$-norm TCP is defined as
\begin{eqnarray}\label{norm p equation}
    \min \|u\|_p^p\\ 
    \mbox{s.t. } u\geq 0,\; \;  \mathcal{A}u^{m-1} + q \geq 0, \; \; u^T (\mathcal{A}u^{m-1} + q)=0.\notag
\end{eqnarray}
\\

\NI In general, the TCP may have many solutions. In this article, we are interested to find the sparse solution of TCP which have the minimum number of nonzero components. For a given tensor $\mathcal{A}$ and a vector $u\in \mathbb{R}^n,$ the problem of seeking a sparse solution can be represented as
\begin{eqnarray}\label{sparse equation}
    \min \|u\|_0\\ 
    \mbox{s.t. } u\geq 0,\; \;  \mathcal{A}u^{m-1} + q \geq 0, \; \; u^T (\mathcal{A}u^{m-1} + q)=0.\notag
\end{eqnarray}

\begin{defn}\cite{bai2016global}
A tensor $\mathcal{A}\in T_{m,n}$ is said to be strong $P$-tensor if for any two different $x=(x_i)$ and $y=(y_i)$ in $\mathbb{R}^n$, $\max_{i\in [n]} (x_i - y_i)(\mathcal{A}x^{m-1} - \mathcal{A}y^{m-1} )_i >0$.
\end{defn}

\begin{defn}\cite{luo2017sparsest}
    Let $\mathcal{A}\in T_{m,n}.$ The tensor $\mathcal{A}$ is said to be a $Z$-tensor if all its off-diagonal entries are nonpositive.
\end{defn}

\begin{theorem}\cite{luo2017sparsest}\label{least element theorem}
    Suppose $\mathcal{A}$ is a $Z$-tensor and $q\in \mathbb{R}^n.$ Suppose that the tensor complementarity problem TCP$(q,\mathcal{A})$ is feasible, i.e., FEA$(q,\mathcal{A}) = \{u \in \mathbb{R}^n : u\geq 0 , \; \mathcal{A}u^{m-1} + q \geq 0 \} \neq \phi .$ Then FEA$(q,\mathcal{A})$ has a unique least element $u^*$ which is also a solution to TCP$(q,\mathcal{A}).$
\end{theorem}

\section{Main results}
Here we explore the properties of sparse solution of TCP$(q,\mathcal{A}).$
\begin{theorem}
Let the solution set of TCP$(q,\mathcal{A})$ be nonempty. Then $\exists$ a sparse solution of the TCP$(q,\mathcal{A}).$
\end{theorem}
\begin{proof}
If SOL$(q,\mathcal{A}) \neq \phi,$ then the feasible region of (\ref{sparse equation}) is nonempty. The objective value of the problem (\ref{sparse equation})  has a finite number of choices among $0,\;1,\; ...,\;n.$ So, (\ref{sparse equation}) is always solvable, and any solution $u^*\in \arg \min_{\Bar{u} \in \mbox{SOL}(q,\mathcal{A})} \|\Bar{u}\|_0$ is a sparse solution of the TCP$(q,\mathcal{A}).$
\end{proof}

\begin{remk}
It is obvious that the sparse solution of TCP$(q,\mathcal{A})$ is unique if SOL$(q,\mathcal{A})$ is unique. For instance, if $\mathcal{A}$ is strong $P$-tensor or strong strictly semipositive tensor then the uniqueness of sparse solution is guaranteed. 
\end{remk}

\NI An element $u$ in $S\subseteq \mathbb{R}^n$ is referred to as the least element of $S$ if and only if $u \leq v$ for all $v$ in $S.$ Under the supposition of the $Z$-tensor, it is possible to obtain the presence of such a least element of set of feasible solutions of TCP$(q,\mathcal{A}).$ Theorem \ref{least element theorem} states about the existance and uniqueness of such least element. Here we establish a connection between the least element of FEA$(q,\mathcal{A})$ and the sparse solution of the TCP involving $Z$-tensor.
\begin{theorem}
   Let a $Z$-tensor $\mathcal{A}\in T_{m,n}$ and $q\in \mathbb{R}^n.$ If FEA$(q,\mathcal{A})\neq \phi,$ then the least element $u$ of FEA$(q,\mathcal{A})$ is a sparse solution of TCP$(q,\mathcal{A}).$
\end{theorem}
\begin{proof}
    From Theorem \ref{least element theorem}, we conclude that the set of feasible solution of TCP$(q,\mathcal{A})$ has a unique least element $y,$ which is also a solution to TCP$(q,\mathcal{A}).$ We assert that this least element $y$ is a sparse solution of TCP$(q,\mathcal{A}).$ If not, let there exists some solution $\Tilde{u} \in$ SOL$(q,\mathcal{A})$ such that $\|\Tilde{u}\|_{0} \leq \|y\|_{0}.$ Then there exists some index $l\in [n]$ such that $y_l\neq 0$ but $\Tilde{u}_l=0.$ This infers that $y_l<0,$ since $y$ is the least element of FEA$(q,\mathcal{A}).$ This contradicts the fact that $y\geq 0.$ Therefore $y$ is a sparse solution of TCP$(q,\mathcal{A}).$
\end{proof}

Here is an example of TCP having infinitely many solution and unique sparse solution.

\begin{examp}
Consider the tensor $\mathcal{A}\in T_{3,2}$ such that $a_{111}=1,\; a_{222}=1.5,\; a_{333}=2,\; a_{131}=-3,\; a_{113}=1,\; a_{133}=-1,\; a_{311}=-2,\; a_{313}=3,\; a_{331}=1.$ Let $q=(-1, 0, 1)^T\in \mathbb{R}^3.$ Then it is easy to find that all vectors $x = (a+ \sqrt{2a^2 +1}, 0, a)^T$ with arbitrary $a \geq 0$ are solutions to the TCP$(q,\mathcal{A}).$ The sparse solution of the TCP$(q,\mathcal{A})$ is $u^* =(1, 0, 0)^T,$ and that is unique.
\end{examp}

\NI Consider the Fischer-Burmeister NCP function defined as 
\begin{equation}\label{FB function}
    \phi_{FB}(u,v)= \sqrt{u^2 +v^2} -(u+v).
\end{equation}
Then $\phi_{FB}(u,v)=0$ if and only if $u\geq 0, \; v\geq 0$ and $uv=0.$
Define a function $\Phi_{FB}: \mathbb{R}^n \to \mathbb{R}^n$ as follows
\begin{equation}\label{FB vector function}
    \Phi_{FB}(u)=\left( \begin{array}{c}
         \phi_{FB}(u_1, (q+ \mathcal{A}u^{m-1})_1) \\
         \vdots\\
         \phi_{FB}(u_n, (q+ \mathcal{A}u^{m-1})_n)
    \end{array} \right).
\end{equation}
Then it is clear from (\ref{ Tensor Complementarity problem}), (\ref{FB function}) and (\ref{FB vector function}) that $u\in$ SOL$(q,\mathcal{A})$ iff $\Phi_{FB}(u)=0.$\\

\NI Now to approximate the sparse solution of TCP$(q,\mathcal{A})$ 
we invoke the $l_p$ regularization and get
\begin{equation}\label{l_p regularized equation}
     \min_{u\in \mathbb{R}^n} f(u)=  \frac{1}{2} \|\Phi_{FB}(u)\|^2 + t \|u\|_p^p,
\end{equation}
where $t \in (0, \infty)$ is a given regularization parameter and $p\in (0,1).$ The unconstrained minimization problem (\ref{l_p regularized equation}) is called $l_p$ regularized minimization problem.

\begin{theorem}
    For any fixed $t>0,$ the solution set of (\ref{l_p regularized equation}) is nonempty and bounded.
\end{theorem}
\begin{proof}
    Let for any fixed $t>0,$ the solution set of (\ref{l_p regularized equation}) be $S_t.$ Note that $f(u)\geq0, \; \forall \; u\in \mathbb{R}^n.$ So the function is bounded below. Again the function is continuous and hence attains the minimum. Therefore the set $S_t$ is nonempty.
    \NI To prove the boundedness of $S_t$ we use the method of contradiction. Let $S_t$ be not bounded. Then there exists a sequence $\{ u_n \}$ in $S_t$ such that $\|u_n\| \to \infty$ as $n \to \infty.$ Consider the sequence of functional values of $\{ u_n \}$ as $\{ f(u_n) \}.$ Then we have $\|f(u_n)\| \to \infty$ as $n \to \infty,$ since $f$ is a coersive function. This contradicts the fact that $u_n \in S_t.$  
\end{proof}

\begin{theorem}\label{minimal l_p norm theorem}
 Let $u_{t}$ be a solution of (\ref{l_p regularized equation}), and $\{t_k\}$ be any positive sequence that converges to $0.$ If SOL$(q, \mathcal{A})$ is nonempty, then $\{u_{t}\}$ has at least one accumulation point, and any accumulation point $u^*$ of $\{u_{t_k}\}$ is a minimal $l_p$ norm solution of TCP$(q, \mathcal{A}),$ i.e., $u^*\in\mbox{SOL}(q,\mathcal{A}) \mbox{ and } \|u^*\|_p^p \leq \|\Bar{u}\|_p^p \mbox{ for any } \Bar{u}\in \mbox{SOL}(q,\mathcal{A}).$
\end{theorem}
\begin{proof}
Suppose $\Bar{u}\in$ SOL$(q,\mathcal{A}).$ Notice that $u_{t_k}$ is a solution of (\ref{l_p regularized equation}) with $t=t_k.$ Also we have
\begin{align}\label{first for first theorem}
    \max\{\frac{1}{2}\|\Phi_{FB}(u_{t_k})\|^2, t_k \|u_{t_k}\|_p^p  \} & \leq \frac{1}{2}\|\Phi_{FB}(u_{t_k})|^2 + t_k \|u_{t_k}|_p^p \notag \\
     & \leq \frac{1}{2}\|\Phi_{FB}(\Bar{u})\|^2 + t_k \|\Bar{u}_{t_k}\|_p^p  \notag \\
     & = t_k \|\Bar{u}_{t_k}\|_p^p.
\end{align}
Now by (\ref{first for first theorem}), it can be easily obtained that, for any $t_k,$
\begin{equation}\label{2nd first theorem}
    \|u_{t_k}\|_p^p \leq \|\Bar{u}\|_p^p =\gamma.
\end{equation}
Therefore the sequence $\{u_{t_k}\}$ is bounded and has at least one accumulation point. Suppose an arbitrary accumulation point of $\{u_{t_k}\}$ is $u^*$ and $\{t_{k_j}\}$ is an subsequence of $t_k$ such that $\lim_{k_j \to \infty} u_{t_{k_j}} = u^*.$

\NI From (\ref{first for first theorem}), for any $t_{k_j},$ we have $\frac{1}{2} \|\Phi_{FB}(u_{t_{k_j}}) \|^2 \leq t_{k_j} \|\Bar{u}\|_p^p.$
Taking $k_j\to \infty$ to both sides of the inequality and by the use of continuity of the function $\|\Phi_{FB}(\cdot)\|$ we obtain
\begin{equation*}
    \|\Phi_{FB}(u^*)\| = \lim_{k_j\to \infty} \|\Phi_{FB}(u_{t_{k_j}})\| =0.
\end{equation*}
This implies $u^*\in $SOL$(q,\mathcal{A}).$ Now from (\ref{2nd first theorem}) for an arbitrary $\Bar{u}\in$ SOL$(q,\mathcal{A})$ we have $\|u_{t_{k_j}}\|_p^p \leq \|u\|_p^p.$
Again by taking $k_j\to \infty$ to the above inequality, we obtain
\begin{equation*}
    \|u^*|_p^p \leq \|u\|_p^p, \mbox{ for any } u\in \mbox{ SOL}(q,\mathcal{A}).
\end{equation*}
Therefore, $u^*$ is a minimal $l_p$ of TCP$(q,\mathcal{A}).$
\end{proof}

\begin{corol}
    Let $\mathcal{A}$ be a $Z$ tensor in Theorem \ref{minimal l_p norm theorem}. Then $u^*$ in Theorem \ref{minimal l_p norm theorem} is a sparse solution of TCP$(q,\mathcal{A}).$
\end{corol}
\begin{proof}
    Let $y$ be the least element of FEA$(q,\mathcal{A}).$ Then by Theorem (\ref{minimal l_p norm theorem}), $y$ is a sparse solution of TCP$(q,\mathcal{A}).$ We show that $u^*$ is the least element of FEA$(q,\mathcal{A}).$ Suppose $u^* \neq y.$ Then $\exists$ atleast one index $l$ such that $u_l^* > y_l \geq 0.$ Therefore we get,
    \begin{align*}
        \|u^*\|_p^p & = \sum_{i\neq l} |u^*_i|^p +|u^*_l|^p\\
        & > \sum_{i\neq l} |y_i|^p + |y_l|^p\\
        & = \|y\|^p,
    \end{align*}
    which contradicts the fact that $u^*$ is a minimal $l_p$ norm solution of TCP$(q,\mathcal{A}).$ Hence $u^* = y,$ and since $u$ is a sparse solution of TCP$(q,\mathcal{A})$ so is $u^*.$
\end{proof}

\NI Let $S_p^*$ denote the set containing local minimizers of (\ref{l_p regularized equation}). Then we have the following theorems.

\begin{theorem}\label{Solution set S^* is bounded}
    The set $S_p^*$ containing the local minimizers of (\ref{l_p regularized equation}) is nonempty and bounded.
\end{theorem}
\begin{proof}
Note that for any fixed $t>0,$ it can be shown that the objective function
$f$ of (\ref{l_p regularized equation}) is coercive which refers to the property that $f (u)\to +\infty$ as $\|u\|\to \infty.$
Since the function $f$ is continuous and bounded below so it attains its minimum. This implies the nonemptyness of the set $S_p^*.$
\end{proof}

\begin{theorem}
    Let $\mathcal{A}\in T_{m,n}$ and the set SOL$(q,\mathcal{A})$ is bounded with an upper bound $B>0.$ Then the set $S_p^*$ is bounded with an upper bound $\Bar{B}= \left(n^{(\frac{1}{p} - \frac{1}{2})}B\right).$
\end{theorem}
\begin{proof}
    Let the SOL$(q,\mathcal{A})$ be bounded and $B>0$ be an upper bound of SOL$(q,\mathcal{A}),$ i.e. $\|u\| \leq B,\; \forall \; u\in$ SOL$(q,\mathcal{A}).$ For $p< 1< 2$ we have 
    \begin{equation*}
        \|u\|_2 \leq \|u\|_p \leq n^{(\frac{1}{p} - \frac{1}{2})} \|u\|_2.
    \end{equation*}
    Then for any $u\in$ SOL$(q,\mathcal{A})$ we have 
    \begin{equation*}
        \frac{1}{2} \|\Phi_{FB}(u)\|^2 = t \|u\|_p^p \leq t \Bar{B}^p.
    \end{equation*}
    Let $\Bar{B}$ be not an upper bound of $S_p^*.$ Then  $\exists \; y \in S_p^*$ such that $f(y)=\min_{u\in \mathbb{R}^n}f(u)$ and $\|y\| \geq \Bar{B}.$ Now
    \begin{align*}
        f(y) & = \frac{1}{2} \|\Phi_{FB}(u)\|^2 + t \|y\|_p^p\\
        & \geq  t \|y\|_p^p\\
        & > t \Bar{B}^p.\\
    \end{align*}
    This contradicts the fact that $y$ is a local minimizer of $(\ref{l_p regularized equation}).$
\end{proof}

Let $\Tilde{q} = -\min \{0, q\}$ where the minimum is taken componentwise. Then relating to the parameter $k,$ we have the following theorem.

\begin{theorem}
    Let $u^*$ be a local minimizer of (\ref{l_p regularized equation}) that satisfies $f(u^*) \leq f(u_0)$ for any given arbitrary initial point $u_0$. If $t \leq \frac{2 \|\Tilde{q}\|^2}{\|\Bar{u}\|_p^p}$ for some nonzero vector $\Bar{u}\in$ SOL$(q,\mathcal{A})$ then $0$ is not a global minimizer of (\ref{l_p regularized equation}).
\end{theorem}
\begin{proof}
Observe that $f(0) = \frac{1}{2} \|\Phi_{FB}(0)\|^2 = 2 \|\Tilde{q}\|^2.$ Also since $\Bar{u}\in$ SOL$(q,\mathcal{A})$, which is nonzero.
\begin{equation}
    f(\Bar{u})= \frac{1}{2}\|\Phi_{FB}(\Bar{u})\|^2 + t \|\Bar{u}\|_p^p =t \|\Bar{u}\|_p^p >0.
\end{equation}    
Thus by the assumption,
\begin{equation}
    f(0)= 2 \|\Tilde{q}\|^2 \geq t \|\Bar{u}\|_p^p =f(\Bar{u}) >0.
\end{equation}
Clearly, $\Bar{u}$ is not a stationary point of (\ref{l_p regularized equation}) since $\Phi_{FB}(\Bar{u})=0.$ Therefore $\exists\; \Tilde{u}$ in a neighborhood of $\Bar{u}$ for which $f (\Bar{u})> f(\Tilde{u}).$ Hence $0$ fails to be a global minimizer of (\ref{l_p regularized equation}).
\end{proof}

\NI Now we consider the $l_p$ regularized model (\ref{l_p regularized equation}) to approximate the sparse solution of TCP in the case of a semi-symmetric tensor. If $\mathcal{A}$ is a semi-symmetric tensor then $\nabla \mathcal{A} u^{m-1} = (m-1) \mathcal{A} u^{m-2}.$ Let define a function $\Psi_{FB} : \mathbb{R}^n \to \mathbb{R}_+$ by $\Psi_{FB}(u)=\frac{1}{2}\|\Phi_{FB}(u)\|^2.$

\begin{lemma}\cite{facchinei2007finite}\label{facchini lemma}
The function $\Psi_{FB}(u)$ is continuously differentiable everywhere and the gradient vector of $\Psi_{FB}(u)$ is given by
\begin{equation}\label{equation from facchini-pang}
    \nabla \Psi_{FB}(u) = \left[D_v(u) + \nabla (\mathcal{A}u^{m-1}) D_z(u) \right] \Phi_{FB}(u)
\end{equation}
where $D_v(u)= diag(v_1(u),...,v_n(u))$ and $D_z(u)= diag(z_1(u),...,z_n(u))$ are two diagonal matrices with diagonal elements given by\\
\[v_i(u)=\left\{\begin{array}{ll}
    \frac{u_i}{\sqrt{u_i^2 + (\mathcal{A}u^{m-1} + q)_i^2}} -1 & \mbox{ if } (u_i, (\mathcal{A}u^{m-1} + q)_i)\neq 0, \\
    \rho -1, & \mbox{ otherwise },
\end{array} \right.\]
\[z_i(u)=\left\{\begin{array}{ll}
    \frac{(\mathcal{A}u^{m-1} + q)_i}{\sqrt{u_i^2 + (\mathcal{A}u^{m-1} + q)_i^2}} -1 & \mbox{ if } (u_i, (\mathcal{A}u^{m-1} + q)_i)\neq 0, \\
    \xi -1, & \mbox{ otherwise },
\end{array} \right.\]
where $\|(\rho_i, \xi_i)\| \leq 1.$
\end{lemma}

\begin{remk}
    In case of semi-symmetric tensor the matrix $\nabla (\mathcal{A}u^{m-1})$ in (\ref{equation from facchini-pang}) can be replaced by $(m-1) \mathcal{A} u^{m-2}.$ Then the gradient vector of $\Psi_{FB}(u)$ can be obtained by
\begin{equation}\label{equation from facchini-pang 2}
    \nabla \Psi_{FB}(u) = \left[D_v(u) + (m-1) \mathcal{A} u^{m-2} D_z(u) \right] \Phi_{FB}(u),
\end{equation}
\end{remk}

Now we provide a lower bound $L$ for any local minimizer of (\ref{l_p regularized equation}.)

\begin{theorem}
    Let $u^*$ be any local minimizer of (\ref{l_p regularized equation}) satisfying $f(u^*) \leq f(u_0)$ for an initial point $u_0.$ Let $\mathcal{A}$ be symmetric tensor and let $\mu$ be the upper bound of $S_p^*.$ Let,
    $L= \left( \frac{t p}{2\sqrt{2}\left(1 +(m-1)\|\mathcal{A}\|\mu^{(m-2)}\right)} \right)^{\frac{1}{1-p}}.$ Then for any $i\in [n]$ we have $u^*_i \in (-L,L) \implies u_i^*=0.$ Also, the number of nonzero entries in $u^*$ is bounded with $\frac{f(u_0)}{t L^p}.$
\end{theorem}
\begin{proof}
    Since $u^*\in S_p^*,$ $\exists\; \delta>0$ and a neighbourhood $N(u^*)=\{ u: \|u-u^*\| \leq \delta \}$ such that
    \begin{equation}\label{lower bound 1st equation}
        f(u^*) \leq f(u) \mbox{ for any } u\in N(u^*).
    \end{equation}
    Let $[n] = I \cup \Tilde{I}$ where $I=\{i: u^*_i \neq 0\}$ and $\Tilde{I}=\{i: u_i*=0 \}.$
    Clearly, $\|u^*\|_0 =|I|,$ which is the cardinality of $I.$ Let us define a function $h: \mathbb{R}^{|I|} \to \mathbb{R}$ by
    \begin{equation}
        h(w) = \frac{1}{2}\left( \sum_{j\in I} \Phi^2_{FB}(w_j, R_{jI}(\mathcal{A}) w^{m-1} + q_j) + \sum_{j\in I} \Phi^2_{FB}(0, R_{jI}(\mathcal{A})w^{m-1} + q_j) \right) +t \|w\|_p^p,
    \end{equation}
    where
    \begin{equation*}
        R_{jI}(\mathcal{A}) w^{m-1}= \sum_{j_2, ..., j_m \in I } a_{j j_2 ... j_m} w_{j_2} \cdots w_{j_m} \mbox{ and }
    \end{equation*}
    \begin{equation*}
         R_{j\Bar{I}}(\mathcal{A}) w^{m-1}= \sum_{j_2, ..., j_m \notin I} a_{j j_2 ... j_m} w_{j_2} \cdots w_{j_m}.
    \end{equation*}
     For any $w\in \mathbb{R}^{|I|},$ we define a vector $u_w$ by letting $(u_w)_I = w$ and $(u_w)_{\Bar{I}}= 0.$ Then it is trivial that $h(w) = f (u_w).$
    Now our claim is  that $w^* = u^*_I$ is a global minimizer of $h(w)$ in the region $\Omega_1= \{ w: \|w- u^*_I\| \leq \delta \}.$ Otherwise there exists $\Tilde{w} \in \Omega_1$ such that $h(\Tilde{w}) < h(w^*).$ We have $u_{\Tilde{w}}\in N(u^*)$ since,
    \[ \|u_{\Tilde{w}} - u^*\| = \|\Tilde{w} - w^*\| = \|\Tilde{w} - u^*_I\| \leq \delta. \]
    By observing that $u_{w^*}=u^*,$ we get$f(u_{\Tilde{w}} ) = h(\Tilde{w} ) < h(w^*)= f(u^*)$ which contradicts (\ref{lower bound 1st equation}). Hence $w^* = u^*_I$ is a local minimizer of 
    \begin{equation}\label{lower bound 2nd equation}
        \min_{w\in \mathbb{R}^{|I|}} h(w)
    \end{equation}
    where $h$ is a continuously differentiable function. Now using the first order necessary optimality condition for (\ref{lower bound 2nd equation}) we have $\nabla h(w^*) = 0.$ Which gives
    \begin{equation}\label{lower bound 3rd equation}
        (\nabla \Psi_{FB}(u^*))_i + t p |u^*_i|^{p-1} sign(u_i^*) = 0,\; \forall \; i\in I.
    \end{equation}
    From the relation $ \frac{1}{2}\| \Phi_{FB}(u^*)\|^2 \leq f(u^*) \leq f(u_0),$ we have 
    \begin{equation}\label{lower bound 4th equation}
         \|\Phi_{FB}(u^*)\| \leq \sqrt{ 2 f(u_0)}.
    \end{equation}
     Now combining Lemma \ref{facchini lemma}, Equation \ref{lower bound 3rd equation}, Equation \ref{lower bound 4th equation} and the Cauchy–Schwarz inequality, we obtain for any $i\in I,$
     \begin{align}\label{lower bound helping equation}
         t p |u^*_i|^{p-1} & = |(\nabla \Psi_{FB} (u^*))_i| \notag \\
         & \leq \|v_i(u) + (m-1) \mathcal{A} u^{m-2} z_i(u) \|\cdot \|\Phi_{FB}(u^*)\| \notag \\
         & \leq 2 \sqrt{2} \left(1 + (m-1)\|\mathcal{A}\| \mu^{(m-2) }\right) \sqrt{f(u_0)}.
     \end{align}
    Note that $p-1 < 0.$ Now using (\ref{lower bound helping equation}), we conclude that
    \[ |u^*_i| \geq  \left( \frac{t p}{2\sqrt{2}(1 + (m-1)\|\mathcal{A}\| \mu^{(m-2)}) \sqrt{f(u_0)}} \right)^{\frac{1}{1-p}} =L, \;\forall \;i\in I. \]
    This means that all nonzero components of $u^*$ are not less than $L.$ Or we can say, for any $i\in [n], \; u^* \in (-L, L) \mbox{ implies } u^*_i=0.$

\NI Now we show the second part of the theorem. By what precedes, $|u^*_i| \geq L$ for any $i\in I.$ Thus 
\begin{equation}\label{lower bound 5th equation}
    \|u^*\|_p^p = \sum_{i\in I} |u^*_i|^p \geq |I| L^p = \|u^*\|_0 L^p
\end{equation}
Thus combining $ t \|u^*\|_p^p \leq \frac{1}{2}\| \Phi_{FB}(u^*)\|^2 + t \|u^*\|_p^p = f(u^*) \leq f(u_0),$ we obtain
\begin{equation}
    \|u^*\|_0 \leq \frac{f(u_0)}{t L^p}.
\end{equation}
\end{proof}

\begin{theorem}
    Let $\mathcal{A}$ be a semi-symmetric tensor. For the regularized model (\ref{l_p regularized equation}), $u^*$ is a local minimizer satisfying $f(u^*) \leq f(u_0)$ for an initial point $u_0.$ Denote,
\[ \gamma (k) = k^{p-1} \left( \frac{2\sqrt{2}(1 + (m-1)\|\mathcal{A}\| \mu^{(m-2)})}{p} \right)^p \left( \sqrt{f(u_0)} \right)^{2-p}, \; k\in [n]. \]
Then the following cases hold:\\
\NI (a) For $k\in [n]$ if $t \geq \gamma(k),$ then $\|u^*\|_0 <k.$\\
\NI (b) For $k=1,$ if $t \geq \gamma(1),$ then $u^* = 0$ is the unique minizer of (\ref{l_p regularized equation}).
\end{theorem}
\begin{proof}
     By (\ref{lower bound 3rd equation}) we conclude for any $0\neq u^* \in S^*_p,$ we have $\Phi_{FB}(u^*)\neq 0.$ According to (\ref{lower bound 5th equation}) and the definition of $L,$ we have
    \begin{align*}
        f(u^*)> t \|u^*\|_p^p & \geq t \|u^*\|_0 L^p\\
        &\geq  t \|u^*\|_0 \left( \frac{t p}{2\sqrt{2}(1 + (m-1)\|\mathcal{A}\| \mu^{(m-2)}) \sqrt{f(u_0)}} \right)^{\frac{p}{1-p}}
    \end{align*}
    \NI(a) When $u^* = 0$ the statement obviously holds. We consider the case $u^* \neq 0.$ For $k\in [n],$ if $t \geq \gamma(k),$ let $\| u^* \| \geq k \geq 1 .$ Then we obtain,
    \begin{align*}
        f(u^*) & > t k \left( \frac{t p}{2\sqrt{2}(1 + (m-1)\|\mathcal{A}\| \mu^{(m-2)}) \sqrt{f(u_0)}} \right)^{\frac{p}{1-p}}\\
        & = t^{\frac{1}{1-p}} k \left( \frac{ p}{2\sqrt{2}(1 + (m-1)\|\mathcal{A}\| \mu^{(m-2)}) \sqrt{f(u_0)}} \right)^{\frac{p}{1-p}}\\
        & \geq \gamma (k)^{\frac{1}{1-p}} k \left( \frac{ p}{2\sqrt{2}(1 + (m-1)\|\mathcal{A}\| \mu^{(m-2)}) \sqrt{f(u_0)}} \right)^{\frac{p}{1-p}} \\
        & = f(u_0),
    \end{align*}
    which contradicts the fact that $f(u^*) \leq f(u_0).$ Hence and hence we obtain $\|u^*\|_0 < k.$\\
    \NI(b) Now for $k = 1,$ if $t \geq \gamma(1),$ then from the result (a) we are getting that $\|u^*\|_0 =0$ for an arbitrary $u^* \in s^*_p.$ This implies that $u^* = 0$ uniquely minimizes (\ref{l_p regularized equation}).
\end{proof}

\section{Conclusion}
In this article, for $\mathcal{A} \in Z$ tensor, we find a relationship between the least element of FEA$(q,\mathcal{A})$ and the sparse solution of TCP$(q,\mathcal{A})$. We consider an unrestricted $p$ norm regularised minimization problem. We show that the sparse solutions of the TCP are approximated by this minimization problem. We provide a lower bound of every nonzero entry of local minimizers. In our proposed regularised model, we also find a lower bound of the regularisation parameter to achive zero global minimizer.

\section{Acknowledgment}
The author R. Deb is thankful for financial support to the Council of Scientific $\&$ Industrial Research (CSIR), India, Junior Research Fellowship scheme.

\bibliographystyle{plain}
\bibliography{referencesall}

\end{document}